\documentclass[11pt]{article}

\usepackage[margin=1in]{geometry}
\usepackage{amsmath,amssymb,amsthm,mathtools}
\usepackage{enumitem}
\usepackage{hyperref}
\usepackage[nameinlink,capitalize]{cleveref}

\newtheorem{theorem}{Theorem}[section]
\newtheorem{lemma}[theorem]{Lemma}
\newtheorem{corollary}[theorem]{Corollary}
\newtheorem{proposition}[theorem]{Proposition}
\theoremstyle{remark}
\newtheorem{remark}[theorem]{Remark}
\newtheorem{definition}[theorem]{Definition}
\theoremstyle{plain}

\usepackage[textsize=footnotesize]{todonotes}

\newcommand{\F}{\mathbb{F}}
\newcommand{\Fp}{\mathbb{F}_\ell}
\newcommand{\Fps}{\mathbb{F}_\ell^{\!*}}
\newcommand{\Z}{\mathbb{Z}}

\title{\huge\bfseries Multi-target hyperbolic sieves and elliptic trace obstructions}
\author{\Large\bfseries  Pantelimon St\u{a}nic\u{a}$^1$, Erik Mulder$^2$, Markus Hittmeir$^3$ 
\vspace{.2cm}\\
$^1$Applied Mathematics Department, Naval Postgraduate School\\
Monterey, CA 93943, USA; \texttt{pstanica@nps.edu}\\
$^2$Department of Informatics, University of Bergen\\
Bergen, Norway; \texttt{erik.mulder@uib.no}\\
$^3$NORCE Analytics, NORCE Research AS, Bergen, Norway;
\texttt{mahi@norceresearch.no}}
\date{}

\begin{document}
\maketitle

\begin{abstract}
Let $N=pq$ be a product of two primes, and let $\ell$ be an odd prime not
 dividing $N$. For $a\in\mathbb F_\ell^*$, the set
\[
 H_a(N;\ell)=\{ax+Nx^{-1}:x\in\mathbb F_\ell^*\}
\]
contains the residue of $ap+q$ modulo $\ell$. Since residues outside this set 
can be excluded when searching for that linear form, it has been applied as 
a sieve for linear combinations of $p$ and $q$ in recent deterministic integer 
factorization methods. Its size 
is $(\ell+\chi(aN))/2$, where $\chi$ is the Legendre symbol modulo $\ell$.
We study unions of these sets, which give a common set of 
candidate residues for several linear forms and thus allow to sieve for multiple 
target forms at once. For distinct $a,b$, we obtain 
an exact formula for
$|H_a(N;\ell)\cup H_b(N;\ell)|$ in terms of
\[
 K(A,B)=\sum_{z\in\mathbb F_\ell}\chi((z^2-A)(z^2-B)),
 \qquad A=4aN,\quad B=4bN.
\]
The Hasse--Weil bound gives an error of $O(\sqrt\ell)$ from $3\ell/4$.
For fixed $c\ne0,1$, we prove the elementary identity
$K(A,cA)=-1-\chi(A)t_c$, where
$t_c=-1-K(1,c)$. Hence the union for a primitive-root pair
$(a,b)=(\gamma,\gamma^2)$ has at most two sizes as $N$ varies.
We give an exact condition on $t_\gamma$ for the bound $3\ell/4+1$ to hold
for every nonzero $N$ modulo $\ell$, and compute the first and second moments.
For $k$ distinct coefficients, we prove
\[
 \left|\left|\bigcup_{j=1}^k H_{a_j}(N;\ell)\right|
       -\ell(1-2^{-k})\right|
 \le (k-1+2^{-k})\sqrt\ell+k.
\]
We derive the corresponding bounds modulo products of distinct primes using
the Chinese remainder theorem. We also determine the exact size of the
image obtained by restricting $x$ to the nonzero $r$-th powers, which applies
when searching for a divisor of the form $u^r$ in $N=u^rv$.
\end{abstract}

\noindent
The second and third authors were supported by the Research Council of Norway (grant 357539).



\section{Introduction}
\label{sec:intro}

Throughout the paper, $\ell>2$ denotes an odd prime, the field $\Fp$ is
identified with $\mathbb Z/\ell\mathbb Z$, and $\Fps=\Fp\setminus\{0\}$ is
its multiplicative group.  The symbol $\chi=\chi_\ell$ denotes the Legendre
symbol modulo $\ell$, extended by $\chi(0)=0$.  Thus $\chi(a)=1$ if
$a\in\Fps$ is a square and $\chi(a)=-1$ otherwise. A nonzero square is
also called a quadratic residue; a nonsquare is a quadratic nonresidue.  Unless explicitly stated
otherwise, $N$ denotes an integer whose image modulo $\ell$ lies in $\Fps$.
A primitive root modulo $\ell$ is a generator of the cyclic group
$\Fps$.

Suppose $N=pq$ is a semiprime, meaning a product of two primes. Knowing
$p+q$ determines the factors: they are the roots of
$X^2-(p+q)X+N$. More generally, if a positive integer coefficient $a$ and
$z=ap+q$ are known, then $p$ is a root of $aX^2-zX+N$.
This motivates restricting the possible values of such linear forms before
searching for them. Deterministic algorithms with rigorous running-time
bounds include the baby-step--giant-step method of
\cite{Hittmeir2018}, the time-space tradeoff in \cite{Hittmeir2021}, and
the exponent-one-fifth algorithms in
\cite{Harvey2021,HarveyHittmeir2022}. The latter give bounds of the form
$O(N^{1/5+o(1)})$. The present paper studies finite-field restrictions that
may be used in searches for linear combinations of the factors.

Let $\ell\nmid N$ be an odd prime. The residues of $p,q$ belong to
\[
 \mathcal H(N;\ell)=\{(x,y)\in\Fps\times\Fps:xy=N\},
\]
which we call the modular hyperbola. For $a\in\Fps$, its image under
$(x,y)\mapsto ax+y$ is
\begin{equation}\label{eq:Ha-def}
 H_a(N;\ell)=\{ax+Nx^{-1}:x\in\Fps\}\subseteq\Fp.
\end{equation}
Taking $x=p\bmod\ell$ shows that $ap+q\bmod\ell$ lies in this set.
Here, an integer coefficient $a$ is reduced modulo $\ell$ when it occurs in
the notation $H_a$.
We use the term \emph{hyperbolic sieve} for the exclusion of residues
outside $H_a$. This is a necessary condition on $ap+q$, not a sufficient
condition for a candidate to come from the factors of $N$.
Lemma~\ref{lem:one-target-size} gives
$|H_a(N;\ell)|=(\ell+\chi(aN))/2$, so approximately half the residues are
excluded. This construction was studied in
\cite{Hittmeir2018,Hittmeir2023}.

Let $m=\ell_1\cdots\ell_n$ be a product of distinct odd primes not dividing
$Na$. The Chinese remainder theorem (CRT) identifies residues modulo $m$
with tuples of residues modulo the $\ell_i$. Consequently the candidate set
\[
 S_a(N;m)=\{z\in\Z/m\Z:z\bmod\ell_i\in H_a(N;\ell_i)
                    \text{ for every }i\}
\]
contains $ap+q\bmod m$ and has size
\[
 |S_a(N;m)|=\prod_{i=1}^n\frac{\ell_i+\chi_{\ell_i}(aN)}2.
\]
Thus restrictions at several primes reduce the number of candidate residues
in a precisely computable way. In the baby-step--giant-step method of
\cite{Hittmeir2018}, one searches for $p+q$ by comparing two lists of modular
values for a common value. The candidate residues determine part of one
list, so reducing their number reduces that part of the search.

Harvey~\cite{Harvey2021} and Harvey--Hittmeir~\cite{HarveyHittmeir2022}
use searches involving linear combinations of the unknown factors, based on
Lehman's method~\cite{Lehman1974}. The scaling formulas in
Section~\ref{sec:scale} explain how sets for many coefficients can be
obtained from two fixed sets at each prime. Our bounds describe the sizes of
common candidate sets for these coefficients. An improvement to a
factorization algorithm would also require bounding the cost of constructing
and using the sets; we do not prove such an improvement here.

For distinct $a,b\in\Fps$, define
\[
 U_{a,b}(N;\ell)=H_a(N;\ell)\cup H_b(N;\ell).
\]
Both $ap+q$ and $bp+q$, reduced modulo $\ell$, lie in this common set.
A residue outside the union can therefore be excluded for either form.
The union does not record which form a residue belongs to; its purpose is to
provide one candidate set that can be used for both. Although each individual
set has about $\ell/2$ elements, this alone gives no useful bound below
$\ell$ for the union. We prove that the union has size
$3\ell/4+O(\sqrt\ell)$ for every pair of distinct nonzero coefficients.
It therefore excludes about one quarter of the residues when $\ell$ is large.

Of particular interest is $a=\gamma$, $b=\gamma^2$, where $\gamma$ is a
primitive root modulo $\ell$. The scaling formulas in
Section~\ref{sec:scale} show how these two sets can be used for every nonzero
coefficient. We also determine exactly when the stronger bound
\begin{equation}\label{eq:conjecture}
 |U_{\gamma,\gamma^2}(N;\ell)|\le\frac{3\ell}{4}+1
\end{equation}
holds. It fails for some choices of $\ell,N,\gamma$.

At several primes, the CRT product of these unions contains the residues
of $ap+q$ for each coefficient class $a\bmod m$ whose residue at $\ell_i$
is either $\gamma_i$ or $\gamma_i^2$. There are $2^n$ such coefficient
classes. Its size is the product of the local union sizes, rather than the
sum of the sizes of $2^n$ separately constructed candidate sets.
Corollary~\ref{thm:CRT} states this precisely. To obtain a reduction below
$m$, it suffices to use local unions smaller than their respective fields;
the stronger bound~\eqref{eq:conjecture} is not needed for that purpose.

The exact formula involves
\begin{equation}\label{eq:K-def}
 K(A,B)=\sum_{z\in\Fp}\chi((z^2-A)(z^2-B)),
 \qquad A=4aN,\quad B=4bN.
\end{equation}
For $c\in\Fps\setminus\{1\}$, define
\[
 t_c=-1-K(1,c).
\]
This integer can be computed directly from the displayed sum. Section~\ref{sec:prelim}
explains its interpretation as a Frobenius trace, meaning $\ell+1$ minus
the number of points on the completed curve $Y^2=(X^2-1)(X^2-c)$.
The later formulas need only this definition and the bound
$|t_c|\le2\sqrt\ell$.

For clarity, we also introduce the ratio notation used in the trace formulas.
For a nonsquare $c\in\Fps$ and $A\in\Fps$, write
$K_c(A)=K(A,cA)=\sum_{z\in\Fp}\chi((z^2-A)(z^2-cA))$.  For the
primitive-root pair $(\gamma,\gamma^2)$, the ratio is $c=\gamma$ and the
relevant parameter is $A=4\gamma N$.

The correction term appearing in the primitive-root formula is
\[
  E(A,B) := r_B\,\chi(B-A) + r_A\,\chi(A-B),
  \qquad r_A := 1 + \chi(A),\quad r_B := 1+\chi(B).
\]
Note $r_A, r_B \in \{0, 2\}$, so $|E(A,B)|\le4$; when
$a = \gamma$ is a nonsquare and $b = \gamma^2$ is a square, exactly one of
$r_A, r_B$ is $2$, forcing $E(A,B)\in\{-2,2\}$.

Our first result is an exact inclusion--exclusion formula for two sets.
For any distinct $a,b\in\Fps$, with $A=4aN$ and $B=4bN$, we prove
\[
  4\,|\Fp\setminus U_{a,b}| = \ell + 2 - r_A - r_B + K(A,B)
                               + r_B\chi(B-A) + r_A\chi(A-B).
\]
For the primitive-root specialization $a=\gamma$, $b=\gamma^2$, this becomes
\[
  |U_{\gamma,\gamma^2}| = \frac{3\ell - K(A,B) - E(A,B)}{4}.
\]
Together with Hasse--Weil, it gives
\[
  \left||U_{a,b}(N;\ell)| - \frac{3\ell}{4}\right|
  \le \frac{2\sqrt\ell + 7}{4}
\]
for distinct $a,b$.

For the primitive-root pair, put $t_\gamma=-1-K(1,\gamma)$. Then $|U_{\gamma,\gamma^2}(N;\ell)|$ assumes at most two values as $N$ varies, according to whether $A=4\gamma N$ is a square or a nonsquare.  Consequently, the natural bound
\[
  |U_{\gamma,\gamma^2}(N;\ell)|\le \frac{3\ell}{4}+1
  \quad\text{for all }N\in\Fps
\]
is equivalent to the trace interval
\[
  -3+2\chi(\gamma-1)
  \le t_\gamma\le
  3+2\chi(1-\gamma).
\]
This criterion explains the counterexamples at $(\ell,N,\gamma)=(5,1,2)$ and
$(13,1,7)$ and turns the original pointwise problem into the problem of choosing a primitive root whose trace lies in this interval.

The same trace distribution gives the first and second moments.  For fixed
nonsquare $c\in\Fps$,
\[
  K(A,cA)=-1-\chi(A)t_c,
\]
where $t_c$ is the Frobenius trace of $Y^2=(X^2-1)(X^2-c)$.  Hence
\[
  \frac{1}{\ell-1}\sum_{A\in\Fps}K(A,cA)^2=t_c^2+1,
\]
with variance $t_c^2\le4\ell$ (the mean squared deviation from the average).  The dependence on $t_c$ is essential; there is
no universal second moment independent of $c$ and $\ell$.

For $k$ distinct coefficients $a_1,\ldots,a_k\in\Fps$, the same finite-field
method gives the explicit estimate for a union of $k$ sets
\[
  \left|\,\Bigl|\bigcup_{j=1}^k H_{a_j}(N;\ell)\Bigr|
       - \ell\Bigl(1-2^{-k}\Bigr)\right|
  \le (k-1+2^{-k})\sqrt\ell+k.
\]
Over $n$ primes $\ell_i \ge L$, the CRT product has size at most
\[
  m\cdot\bigl(1 - 2^{-k} + (k-1+2^{-k})L^{-1/2}+kL^{-1}\bigr)^n
\]
and simultaneously covers all $k^n$ coefficient classes specified by the local choices.

Finally, let $r\ge2$ and write $(\Fps)^r=\{x^r:x\in\Fps\}$. For inputs $N=u^rv$ and primes $\ell\equiv1\pmod r$,
the local hyperbola restricted to $r$-th powers $P_r(N;\ell)=\{(x,y) \in \Fps\times\Fps : xy \equiv N \pmod\ell,\;  x \in (\Fps)^r\}$
has size exactly
$(\ell-1)/r$, giving an exact CRT saving by $r^n$.  
We also determine the
exact size of the restricted linear image
\[
  H_{a,r}(N;\ell)=\{ax+Nx^{-1}:x\in(\Fps)^r\}
\]
by counting pairs of domain elements with the same image.

Section~\ref{sec:prelim} states the character-sum and point-count
estimates. Section~\ref{sec:two-target} studies the individual images,
their scaling formulas, and unions of two sets, including the primitive-root
case and its counterexamples. Section~\ref{sec:second-moment} gives the
averages and fluctuation bounds. Section~\ref{sec:k-target} treats unions
of several sets, the associated CRT products, and images restricted to
power subgroups. Section~\ref{sec:algorithmic} discusses exact computation
and coefficient choices. Section~\ref{sec:conclusion} gives the conclusions
and further questions. The appendices provide verification scripts.

\section{Preliminaries}
\label{sec:prelim}

This section states the character-sum estimates used in the proofs.
A polynomial is \emph{squarefree} if it has no repeated root in an algebraic
closure of $\Fp$. For a finite set $S$, both $|S|$ and $\#S$ denote its
number of elements. An indicator $\mathbf 1[P]$ is $1$ when the condition
$P$ holds and $0$ otherwise.

For the curve estimate, we describe explicitly which points are counted.
By the completion of $Y^2=f(X)$ we mean the curve obtained by adding
its points at infinity. For the squarefree quartics below, this completion
has no singular points, that is, no point at which both partial derivatives
of a local defining equation vanish. This is the nonsingular completion
used in the Hasse--Weil bound.
If $f$ is a squarefree quartic with square leading coefficient, the equation
$Y^2=f(X)$ has two points at infinity in addition to its solutions
$(X,Y)\in\Fp^2$. Write $E(\Fp)$ for this completed set of points and define
its \emph{Frobenius trace} by $t_E=\ell+1-\#E(\Fp)$.
The Hasse--Weil bound applies to this completion.
The later arguments use the point counts and character sums explicitly.
The estimates needed are Theorems~\ref{thm:weil-character-sum}
and~\ref{thm:hasse-weil-genus-one}.

\begin{lemma}\label{lem:quadratic-shift-sum}
For every $A\in\Fps$,
\[
  \sum_{z\in\Fp}\chi(z^2-A)=-1.
\]
More generally, if $f(X)=uX^2+vX+w\in\Fp[X]$ has $u\ne0$ and nonzero
discriminant, then $\sum_{x\in\Fp}\chi(f(x))=-\chi(u)$.
\end{lemma}

\begin{proof}
For the first identity, count the points on $Y^2=z^2-A$.  The equation is
$(z-Y)(z+Y)=A$, and every choice of a nonzero value for $z-Y$ uniquely
determines $z+Y$; hence it has exactly $\ell-1$ affine solutions.  On the
other hand, the number of such solutions is
$\sum_{z\in\Fp}(1+\chi(z^2-A))=\ell+\sum_z\chi(z^2-A)$, giving the claim.
For the general formula, let $D=v^2-4uw\ne0$. Completing the square gives
\[
 f(x)=u\left(\left(x+\frac{v}{2u}\right)^2-\frac{D}{4u^2}\right).
\]
Translation by $v/(2u)$ permutes $\Fp$, and $D/(4u^2)\ne0$. Thus
multiplicativity of $\chi$ and the first identity give
\[
 \sum_{x\in\Fp}\chi(f(x))
 =\chi(u)\sum_{z\in\Fp}\chi\left(z^2-\frac{D}{4u^2}\right)
 =-\chi(u).
\]
See also \cite[Ch.~5]{LidlNiederreiter1997}.
\end{proof}

The following is the Weil bound for quadratic character sums. We use it
only for squarefree polynomials.  A standard reference is \cite[Thm.~5.41]{LidlNiederreiter1997}.
\begin{theorem}\label{thm:weil-character-sum}
Let $f\in\Fp[X]$ be squarefree of positive degree $d$.  Then
\[
  \left|\sum_{x\in\Fp}\chi(f(x))\right|\le (d-1)\sqrt\ell.
\]
\end{theorem}

We will also need the following form of the Hasse--Weil bound; see~\cite[Ch.~V, Thm.~1.1]
{Silverman2009} or Weil's original work~\cite{Weil1948}.
\begin{theorem}\label{thm:hasse-weil-genus-one}
Let $f\in\Fp[X]$ be squarefree of degree $4$ with square leading coefficient.
Let $E(\Fp)$ consist of the affine solutions of $Y^2=f(X)$ and its two
points at infinity. If $t_E=\ell+1-\#E(\Fp)$, then
\[
 |t_E|\le2\sqrt\ell.
\]
\end{theorem}

\begin{proposition}\label{prop:quartic-trace}
Let $f\in\Fp[X]$ be squarefree of degree $4$ with square leading coefficient,
and use the point-count convention above for $Y^2=f(X)$. Then
\[
 \#E(\Fp)=\ell+2+\sum_{x\in\Fp}\chi(f(x)),
 \qquad t_E=-1-\sum_{x\in\Fp}\chi(f(x)).
\]
\end{proposition}

\begin{proof}
For each $x$, the equation $Y^2=f(x)$ has $1+\chi(f(x))$ solutions.
Their total is $\ell+\sum_x\chi(f(x))$. Adding the two points at infinity
proves the formula. To see these two points explicitly, put $u=1/X$ and
$v=Y/X^2$. The equation becomes $v^2=u^4f(1/u)$, and at $u=0$ its
right-hand side is the square leading coefficient of $f$. Writing $b$ for the leading coefficient, the two choices are
$v=\sqrt b$ and $v=-\sqrt b$. They are distinct and nonzero because
$\ell$ is odd. These are all the added points, since every point with
$X\ne0$ is in the overlap of the two coordinate descriptions.

For completeness, the affine equation has a singular point only if
$Y=0$ and $f(X)=f'(X)=0$; squarefreeness rules this out. In the second
coordinate description, the derivative with respect to $v$ is $2v$,
which is nonzero at either added point. The completion is therefore
nonsingular. A squarefree quartic with a rational point at infinity can
be represented by a nonsingular cubic with a chosen rational point;
such a curve is called an elliptic curve. The Hasse--Weil bound for
elliptic curves consequently applies to this completion; see
\cite[Chs.~III and V]{Silverman2009}.
\end{proof}

\section{Unions of two hyperbolic sieve sets}
\label{sec:one-target}\label{sec:scale}\label{sec:two-target}\label{sec:results}

We begin by making the set for one coefficient completely precise, since its exact
size formula is needed at several points.

\begin{lemma}\label{lem:image-criterion}
For $a, N \in \Fps$ and $z \in \Fp$,
\[
  z \in H_a(N;\ell)
  \quad\Longleftrightarrow\quad
  \chi(z^2 - 4aN) \ne -1.
\]
\end{lemma}

\begin{proof}
The equation $ax + Nx^{-1} = z$ in $x \in \Fps$ is equivalent to
$ax^2 - zx + N = 0$.  Since $aN \ne 0$, this quadratic has discriminant
$\Delta = z^2 - 4aN$ and a solution in $\Fp$ if and only if
$\chi(\Delta) \ne -1$ (i.e., $\Delta$ is zero or a nonzero square).
Any root is nonzero because the constant term $N$ is nonzero.
\end{proof}

\begin{lemma}\label{lem:one-target-size}
For $a, N \in \Fps$,
\[
  |H_a(N;\ell)| = \frac{\ell + \chi(aN)}{2}.
\]
In particular, $|H_a(N;\ell)|$ equals $(\ell-1)/2$ or $(\ell+1)/2$,
according as $aN$ is a nonsquare or square modulo~$\ell$, respectively.
\end{lemma}

\begin{proof}
Let $A=4aN$.  By Lemma~\ref{lem:image-criterion}, the complement of
$H_a(N;\ell)$ consists of the $z\in\Fp$ with $\chi(z^2-A)=-1$.  Put
\[
  n_+=\#\{z:\chi(z^2-A)=1\},\quad
  n_-=\#\{z:\chi(z^2-A)=-1\},\quad
  n_0=\#\{z:z^2=A\}.
\]
Since $A\ne0$, $n_0=1+\chi(A)=1+\chi(aN)$.  Lemma~\ref{lem:quadratic-shift-sum} gives $n_+-n_-=-1$, while $n_++n_-+n_0=\ell$.  Solving for $n_-$ yields
$n_-=(\ell-\chi(aN))/2$.  Hence
\[
  |H_a(N;\ell)|=\ell-n_- = \frac{\ell+\chi(aN)}2. \qedhere
\]
\end{proof}

\begin{remark}
The set for one coefficient $H_a(N;\ell)$ has size exactly $(\ell\pm1)/2$,
with the sign determined by whether $aN$ is a quadratic residue.
For distinct odd primes $\ell_i\nmid Na$, put $m=\ell_1\cdots\ell_n$
and $\epsilon_i=\chi_{\ell_i}(aN)\in\{-1,1\}$. The CRT candidate set
for $ap+q$ then has the exact size
\[
 |S_a(N;m)|=\frac{m}{2^n}\prod_{i=1}^n
       \left(1+\frac{\epsilon_i}{\ell_i}\right)
 \le\frac{m}{2^n}\prod_{i=1}^n\left(1+\frac1{\ell_i}\right).
\]
This expression separates the factor $2^{-n}$ from the corrections at
the selected primes.
\end{remark}

	We provide background for why the union $H_\gamma(N;\ell) \cup 
	H_{\gamma^2}(N;\ell)$ for a primitive root $\gamma$ modulo $\ell$ is of particular 
	interest for potential applications of the sieve. The following lemma first explains why we 
	only 
	consider 
	the linear form $(x,y)\mapsto ax + y$, and not the more general form $(x,y)\mapsto ux + vy$.
	
	\begin{lemma}\label{lem:linearform}
		Let $a, N \in \Fps$. For all $(u,v)\in  \Fps\times \Fps$ satisfying $a = uv$, we have
		\[
		H_a(N;\ell)=\{ux+vNx^{-1}: x\in \Fps\}.
		\]
	\end{lemma}
	
	\begin{proof}
Put $x=vt$. As $x$ ranges over $\Fps$, so does $t$, and
$ux+vNx^{-1}=uvt+Nt^{-1}=at+Nt^{-1}$.
This is also \cite[Lemma 3.2]{Hittmeir2023}.
\end{proof}
	
	\begin{lemma}\label{lem:scale}
Let $a,N \in \Fps$ and $\gamma$ be an arbitrary primitive root modulo  $\ell$. If $\chi(a)=1$, then 
\[
H_a(N;\ell)=\{\theta_a\cdot s \mid s\in H_{\gamma^2}(N;\ell)\},
\]
where $\theta_a$ is a square root of $a\gamma^{-2}$. 
If $\chi(a)=-1$, then 
\[
H_a(N;\ell)=\{\theta_a\cdot s \mid s\in H_{\gamma}(N;\ell)\},
\]
where $\theta_a$ is a square root of $a\gamma^{-1}$.
\end{lemma}
	
\begin{proof}
Since $\gamma$ generates the cyclic group $\Fps$ of even order $\ell-1$,
it is a nonsquare, whereas $\gamma^2$ is a square. Put $\Gamma=\gamma^2$
when $\chi(a)=1$, and $\Gamma=\gamma$ when $\chi(a)=-1$. In both cases
$\chi(a\Gamma^{-1})=1$, so there is $\theta_a\in\Fps$ with
$\theta_a^2=a\Gamma^{-1}$. For every $x\in\Fps$,
\[
 \theta_a(\Gamma x+Nx^{-1})
 =a(x\theta_a^{-1})+N(x\theta_a^{-1})^{-1}.
\]
The substitution $z=x\theta_a^{-1}$ is a bijection of $\Fps$.
Thus the values on the left are exactly all the values defining
$H_a(N;\ell)$, which proves both cases.
Either choice of the square root gives the same image set: replacing
$x$ by $-x$ shows that $H_\Gamma(N;\ell)=-H_\Gamma(N;\ell)$.
\end{proof}

Lemma \ref{lem:scale} allows us to represent $H_a(N;\ell)$ for any $a\in\Fps$ via the sets 	$H_\gamma(N;\ell)$ and $H_{\gamma^2}(N;\ell)$. We now show how to represent  hyperbolic sieve sets modulo products of primes by scaling a fixed set. Write $\Z_m=\Z/m\Z$ and $\Z_m^*$ for its group of invertible residues. We identify Cartesian products of subsets of $\F_{\ell_i}$ with their images in $\Z_m$ under the inverse Chinese remainder isomorphism.

\begin{theorem}\label{thm:scale}
Let $m=\ell_1\ldots\ell_n$ be a product of $n$ distinct odd primes not dividing  $N$, let $a\in\Z_m^*$ and $\gamma_i$ be a primitive root modulo $\ell_i$. Then
\[
\prod_{i=1}^n H_a(N;\ell_i) = \left\{\theta_a\cdot s \mod m \,\bigg|\, s\in	\prod_{i=1}^n H_{\Gamma_i}(N;\ell_i)\right\},
\]
where $\theta_a\in\Z_m^*$ and $\Gamma_i\in\F_{\ell_i}^*$ are chosen such that, for every $i=1,\ldots,n$,	
\[
\begin{aligned}
\theta_a^2 \equiv
\begin{cases}
a\gamma_i^{-2}  \pmod{\ell_i}, & \chi_{\ell_i}(a) = 1,\\
a\gamma_i^{-1} \pmod{\ell_i}, & \chi_{\ell_i}(a) = -1.
\end{cases}
\end{aligned}
\hspace{2cm}
\begin{aligned}
\Gamma_i =
\begin{cases}
\gamma_i^2,   & \chi_{\ell_i}(a) = 1,\\
	\gamma_i, & \chi_{\ell_i}(a) = -1.
\end{cases}
\end{aligned}
\]		
\end{theorem}
\begin{proof}
For each $i=1,\ldots,n$, choose $\theta_{a,i}\in\F_{\ell_i}^*$ satisfying
\[
\theta_{a,i}^2 =
\begin{cases}
a\gamma_i^{-2}, & \chi_{\ell_i}(a)=1,\\
a\gamma_i^{-1}, & \chi_{\ell_i}(a)=-1.
\end{cases}
\]
Let $\theta_a\in\Z_m^*$ be the unique element such that
$\theta_a\equiv \theta_{a,i}\pmod{\ell_i}$ for all $i$. The CRT identification gives
\[\left\{\theta_a\cdot s \mod m \,\bigg|\, s\in	
\prod_{i=1}^n 
H_{\Gamma_i}(N;\ell_i)\right\}=
\prod_{i=1}^n
\left\{
\theta_{a,i}\cdot s_i
\,\middle|\,
s_i\in H_{\Gamma_i}(N;\ell_i)
\right\}
\]
under our Chinese remainder identification. The statement then follows from Lemma~\ref{lem:scale}.
\end{proof}
	
In Theorem \ref{thm:scale} the set for the values $s$ on the
right-hand side depends on $a$. A set independent of $a$ is obtained by 
covering the base set with the Cartesian products of the unions 
$U_{\gamma_i,\gamma_i^2}(N;\ell_i)=H_{\gamma_i}(N;\ell_i)\cup H_{\gamma^2_i}(N;\ell_i)$, 
allowing us to obtain the following two results.
\begin{corollary}
With the same notation as in Theorem \ref{thm:scale},  we have 
\[
\prod_{i=1}^n H_a(N;\ell_i) \subseteq \left\{\theta_a\cdot s \mod m \,\bigg|\, s\in	
\prod_{i=1}^n 
U_{\gamma_i,\gamma_i^2}(N;\ell_i)\right\}.
\]
\end{corollary}
\begin{proof}
For each $i$, the coefficient $\Gamma_i$ in Theorem~\ref{thm:scale}
is either $\gamma_i$ or $\gamma_i^2$. Hence
$H_{\Gamma_i}(N;\ell_i)\subseteq U_{\gamma_i,\gamma_i^2}(N;\ell_i)$.
Taking products and multiplying by $\theta_a$ in $\Z_m$ preserves
containment, giving the result.
\end{proof}

\begin{corollary}
With the same notation as in Theorem \ref{thm:scale} and $\theta\in \Z_m^*$,  we have 
\[
\prod_{i=1}^n H_a(N;\ell_i) \subseteq \left\{\theta\cdot s \mod m \,\bigg|\, s\in	
\prod_{i=1}^n 
U_{\gamma_i,\gamma_i^2}(N;\ell_i)\right\}
\]
for any $a\in\Z_m^*$ that satisfies $a\pmod{\ell_i}\in\{\theta^2\gamma_i \pmod{\ell_i}, 
\theta^2\gamma_i^2 \pmod{\ell_i}\}$ for 
all $i=1,\ldots,n$.
\end{corollary}
\begin{proof}
For each $i$, choose $\Gamma_i\in\{\gamma_i,\gamma_i^2\}$ such that
$a=\theta^2\Gamma_i$ modulo $\ell_i$. With $\theta_i=\theta\bmod\ell_i$,
the substitution used in Lemma~\ref{lem:scale} gives
$H_a(N;\ell_i)=\theta_i H_{\Gamma_i}(N;\ell_i)$.
The latter set is contained in $\theta_i U_{\gamma_i,\gamma_i^2}(N;\ell_i)$.
Taking the product under the CRT gives the displayed containment.
\end{proof}

These results cover hyperbolic sieve sets for many linear 
forms $(x,y)\mapsto ax + y$ with a single superset created from 
$U_{\gamma_i,\gamma_i^2}(N;\ell_i)$. However, the statements are only meaningful if 
$U_{\gamma_i,\gamma_i^2}(N;\ell_i)$ is 
substantially smaller than $\F_{\ell_i}$. The subsequent analysis gives bounds below $\ell$ for sufficiently large primes.

Let us consider the union
$U_{a,b} := H_a(N;\ell) \cup H_b(N;\ell)$ for general, distinct $a, b \in \Fps$.
Set $A = 4aN$ and $B = 4bN$.  By Lemma~\ref{lem:image-criterion}, a residue
$z\in\Fp$ lies outside $U_{a,b}$ if and only if both $z^2-A$ and $z^2-B$
are quadratic nonresidues modulo~$\ell$.  We count these residues by expanding their indicator functions.

\begin{theorem}\label{thm:exact-union}
Let $a, b, N \in \Fps$ with $a \ne b$.  Set $A = 4aN$, $B = 4bN$,
$r_A = 1 + \chi(A) \in \{0,2\}$, $r_B = 1 + \chi(B) \in \{0,2\}$.  Then
\[
  4\,|\Fp \setminus U_{a,b}|
  = \ell + 2 - r_A - r_B + K(A,B) + r_B\chi(B-A) + r_A\chi(A-B),
\]
where $K(A,B) = \sum_{z \in \Fp}\chi\bigl((z^2-A)(z^2-B)\bigr)$.
\end{theorem}

\begin{proof}
For $C \in \{A,B\}$ define $\delta_C(z) = \mathbf{1}[z^2 = C]$; note
$\sum_{z\in\Fp}\delta_C(z) = r_C$ since $z^2=C$ has $1+\chi(C)$ solutions
because $C\ne0$.  By
Lemma~\ref{lem:image-criterion}, $z\notin H_a(N;\ell)$ iff $\chi(z^2-A)=-1$,
i.e., $z^2-A$ is a nonzero nonsquare.  The indicator of this event is
\[
  I_A(z) := \frac{1 - \chi(z^2-A) - \delta_A(z)}{2},
\]
and similarly $I_B(z)$ for $B$.  Then $|\Fp\setminus U_{a,b}| = \sum_z I_A(z)I_B(z)$,
so
\begin{align*}
  4|\Fp\setminus U_{a,b}|
  &= \sum_{z\in\Fp}
     \bigl(1 - \chi(z^2-A) - \delta_A(z)\bigr)
     \bigl(1 - \chi(z^2-B) - \delta_B(z)\bigr).
\end{align*}
Expanding the product, the constant term contributes $\ell$.  By
Lemma~\ref{lem:quadratic-shift-sum}, the two linear character sums satisfy
$-\sum_z\chi(z^2-A)=-\sum_z\chi(z^2-B)=1$.  The two root indicators
contribute $\sum_z\delta_A(z)=r_A$ and $\sum_z\delta_B(z)=r_B$, while the
product of the two characters is exactly $K(A,B)$.  The mixed terms
with one character and one root indicator are supported at the roots of
$z^2=A$ or $z^2=B$, and hence contribute respectively
$r_B\chi(B-A)$ and $r_A\chi(A-B)$.  Finally,
$\sum_z\delta_A(z)\delta_B(z)=0$, since $A\ne B$.  Collecting these terms
gives
\[
  4|\Fp\setminus U_{a,b}|
  = \ell +2-r_A-r_B+K(A,B)+r_B\chi(B-A)+r_A\chi(A-B),
\]
which is the desired formula.
\end{proof}

The character sum $K(A,B)$ can also be read as a point count.  For $C\in\{A,B\}$, the derivative of $X^2-C$ is $2X$. A common
root of $X^2-C$ and $2X$ would force $C=0$, which is excluded. The two
quadratics have no common root because $A\ne B$. Thus the polynomial $f_{A,B}(X)=(X^2-A)(X^2-B)$ is squarefree of degree
$4$ and has leading coefficient $1$.  Proposition~\ref{prop:quartic-trace}
applied to the completed curve $\mathcal E_{A,B}$ given by
$Y^2=f_{A,B}(X)$ gives
\[
  \#\mathcal E_{A,B}(\Fp)=\ell+K(A,B)+2
  \quad\text{and}\quad
  t(\mathcal E_{A,B})=-K(A,B)-1.
\]
The Hasse--Weil estimate in Theorem~\ref{thm:hasse-weil-genus-one} therefore
immediately gives the following bound.

\begin{corollary}\label{cor:hasse-weil}
For all $A, B \in \Fps$ with $A\ne B$,
\[
  |K(A,B) + 1| \le 2\sqrt\ell.
\]
\end{corollary}

\begin{proof}
The preceding paragraph gives $K(A,B)+1=-t(\mathcal E_{A,B})$.  The result is
therefore exactly Theorem~\ref{thm:hasse-weil-genus-one} applied to
$\mathcal E_{A,B}$.
\end{proof}

Specialize to $a = \gamma$, $b = \gamma^2$ where $\gamma$ is a primitive root
modulo $\ell$.  Since $\gamma$ generates $\Fps$, it is a quadratic nonresidue
($\chi(\gamma) = -1$) and $\gamma^2$ is a quadratic residue ($\chi(\gamma^2)=1$).
Hence $r_A = 1+\chi(4\gamma N) = 1+\chi(\gamma)\chi(N) = 1-\chi(N)$ and
$r_B = 1+\chi(\gamma^2)\chi(N) = 1+\chi(N)$, so $r_A + r_B = 2$.

\begin{corollary}\label{cor:primitive-root-formula}
Let $\gamma$ be a primitive root modulo $\ell$, $A = 4\gamma N$,
$B = 4\gamma^2 N$.  Define
\[
  E(A,B) := r_B\,\chi(B-A) + r_A\,\chi(A-B) \in \{-2,2\}.
\]
Then
\[
  |U_{\gamma,\gamma^2}| = \frac{3\ell - K(A,B) - E(A,B)}{4}.
\]
\end{corollary}

\begin{proof}
Substituting $r_A + r_B = 2$ into Theorem~\ref{thm:exact-union} gives
$4|\Fp\setminus U_{a,b}| = \ell + 2 - 2 + K(A,B) + E(A,B) = \ell + K(A,B) + E(A,B)$.
So $|U_{\gamma,\gamma^2}| = \ell - (\ell + K(A,B)+E(A,B))/4 = (3\ell-K(A,B)-E(A,B))/4$.
For $E$: since $\gamma$ is a nonsquare, exactly one of $r_A,r_B$ equals $2$
and the other $0$.  If $r_A=2$ then $E=2\chi(A-B)$; if $r_B=2$ then
$E=2\chi(B-A)$.  In either case $\chi$ is evaluated at the nonzero element
$A-B$ or $B-A$ respectively, giving $E\in\{-2,2\}$.
\end{proof}

The primitive-root case involves $K(A,\gamma A)$. The following identity
shows that its value depends only on whether $A$ is a square.

\begin{theorem}\label{thm:two-valued-K}
Let $c\in\Fps$ with $c\ne1$, and consider
\[
  \mathcal E_c:\quad Y^2=(X^2-1)(X^2-c)
\]
with the completion described in
Section~\textup{\ref{sec:prelim}}. Write $t_c=\ell+1-\#\mathcal E_c(\Fp)$
for its Frobenius trace.  Then, for every
$A\in\Fps$,
\[
  K(A,cA)=-1-\chi(A)t_c.
\]
Equivalently,
\[
  K(A,cA)=
  \begin{cases}
    -1-t_c, & A\text{ a square},\\
    -1+t_c, & A\text{ a nonsquare}.
  \end{cases}
\]
\end{theorem}

\begin{proof}
Each $u\in\Fp$ has $1+\chi(u)$ square roots in $\Fp$. Hence
\[
 K(A,cA)=\sum_{u\in\Fp}(1+\chi(u))\chi((u-A)(u-cA)).
\]
The part without $\chi(u)$ equals $-1$ by
Lemma~\ref{lem:quadratic-shift-sum}, since $A\ne cA$. In the remaining
part, substitute $u=Av$. Multiplicativity of $\chi$ gives
\[
 K(A,cA)=-1+\chi(A)\sum_{v\in\Fp}\chi(v(v-1)(v-c)).
\]
Setting $A=1$ in this equation implies that the last sum is equal to $K(1,c)+1$. Proposition~\ref{prop:quartic-trace} then gives $K(1,c)+1=-t_c$. This proves the stated identity.
\end{proof}

The union size for the primitive-root pair can now be evaluated without averaging over
$N$; only whether $A=4\gamma N$ is a square or a nonsquare remains.

\begin{theorem}
\label{thm:primitive-root-two-values}
Let $\gamma$ be a primitive root modulo $\ell$, set $A=4\gamma N$, and let
$t_\gamma$ be the Frobenius trace of
\[
  \mathcal E_\gamma:\quad Y^2=(X^2-1)(X^2-\gamma).
\]
Then
\[
  |U_{\gamma,\gamma^2}(N;\ell)|
  =
  \begin{cases}
  \displaystyle
  \frac{3\ell+1+t_\gamma-2\chi(1-\gamma)}{4},
  & A\text{ a square},\\[1.2em]
  \displaystyle
  \frac{3\ell+1-t_\gamma+2\chi(\gamma-1)}{4},
  & A\text{ a nonsquare}.
  \end{cases}
\]
Consequently, for fixed $\ell$ and $\gamma$, the union size is determined by whether $A$ is a square or a nonsquare as $N$ ranges over
$\Fps$.
\end{theorem}

\begin{proof}
The primitive-root pair has ratio $c=\gamma$, so Theorem~\ref{thm:two-valued-K}
gives
\[
  K(A,\gamma A)=-1-\chi(A)t_\gamma .
\]
It remains only to evaluate the boundary term
\[
  E(A,\gamma A)=r_{\gamma A}\chi(\gamma A-A)+r_A\chi(A-\gamma A).
\]
Since $\gamma$ is a nonsquare, exactly one of $A$ and $\gamma A$ is a square.
If $A$ is a square, then $r_A=2$ and $r_{\gamma A}=0$, so
\[
  E(A,\gamma A)=2\chi(A(1-\gamma))=2\chi(1-\gamma).
\]
Substituting $K(A,\gamma A)=-1-t_\gamma$ in
Corollary~\ref{cor:primitive-root-formula} gives the first displayed value.
If $A$ is a nonsquare, then $r_A=0$ and $r_{\gamma A}=2$, so
\[
  E(A,\gamma A)=2\chi(A(\gamma-1))=-2\chi(\gamma-1),
\]
and $K(A,\gamma A)=-1+t_\gamma$.  Substitution gives the second displayed
value.
\end{proof}

\begin{corollary}
\label{cor:exact-trace-criterion}
Let $\gamma$ be a primitive root modulo $\ell$, and let $t_\gamma$ be the
Frobenius trace of
\[
  Y^2=(X^2-1)(X^2-\gamma).
\]
Then
\[
  |U_{\gamma,\gamma^2}(N;\ell)|\le \frac{3\ell}{4}+1
  \quad\text{for every }N\in\Fps
\]
if and only if
\[
  -3+2\chi(\gamma-1)
  \le
  t_\gamma
  \le
  3+2\chi(1-\gamma).
\]
\end{corollary}

\begin{proof}
By Theorem~\ref{thm:primitive-root-two-values}, the desired inequality must
hold for the two square classes of $A=4\gamma N$.
For square $A$ it is equivalent to
\[
  \frac{3\ell+1+t_\gamma-2\chi(1-\gamma)}4\le \frac{3\ell}4+1,
\]
that is, $t_\gamma\le 3+2\chi(1-\gamma)$.  For nonsquare $A$ it is equivalent
to
\[
  \frac{3\ell+1-t_\gamma+2\chi(\gamma-1)}4\le \frac{3\ell}4+1,
\]
that is, $t_\gamma\ge -3+2\chi(\gamma-1)$.  Combining the two inequalities
proves the criterion.
\end{proof}

\begin{remark}
\label{rem:trace-verify}
The following two counterexamples violate the trace interval.
Appendix~\ref{app:sage} gives code to check the traces and union formulas.

For $\ell=5$ and $\gamma=2$, we have
$\chi(\gamma-1)=\chi(1)=1$ and $\chi(1-\gamma)=\chi(4)=1$, so
Corollary~\textup{\ref{cor:exact-trace-criterion}} gives the admissible
interval $[-1,5]$.  The curve
\[
  \mathcal{E}_2:\quad Y^2=(X^2-1)(X^2-2)
\]
over $\mathbb{F}_5$ has Frobenius trace $t_\gamma=-2$, which lies outside
this interval.  Hence the pointwise bound
$|U_{\gamma,\gamma^2}(N;\ell)|\le 3\ell/4+1$ cannot hold for all
$N\in\mathbb{F}_5^*$.

For $\ell=13$ and $\gamma=7$, we have
$\chi(\gamma-1)=\chi(6)=-1$ and
$\chi(1-\gamma)=\chi(7)=-1$, so the admissible interval is $[-5,1]$.
The curve
\[
  \mathcal{E}_7:\quad Y^2=(X^2-1)(X^2-7)
\]
over $\mathbb{F}_{13}$ has Frobenius trace $t_\gamma=-6$, again outside
the interval.  Thus the pointwise bound fails for some
$N\in\mathbb{F}_{13}^*$.

The computation returns
\[
\begin{array}{c|c|c|c}
\ell & \gamma & t_\gamma & \text{admissible interval}\\
\hline
5 & 2 & -2 & [-1,5]\\
13 & 7 & -6 & [-5,1]
\end{array}
\]
and verifies the two-case formula of
Theorem~\textup{\ref{thm:primitive-root-two-values}} for every
$N\in\mathbb{F}_\ell^*$ in both cases.
\end{remark}

We next record the uniform bound that follows from the exact formula and the
Hasse--Weil estimate.

\begin{corollary}\label{thm:corrected-bound}
For any prime $\ell > 2$, $N \in \Fps$, and distinct $a, b \in \Fps$,
\[
  \left||U_{a,b}(N;\ell)| - \frac{3\ell}{4}\right| \le \frac{2\sqrt\ell + 7}{4}.
\]
\end{corollary}

\begin{proof}
From Theorem~\ref{thm:exact-union},
\[
  4|\Fp\setminus U_{a,b}| = \ell + 2 - r_A - r_B + K(A,B) + r_B\chi(B-A) + r_A\chi(A-B).
\]
The boundary terms satisfy
$|2-r_A-r_B+r_B\chi(B-A)+r_A\chi(A-B)|\le6$: indeed
$|2-r_A-r_B|\le2$ because $r_A,r_B\in\{0,2\}$, while
$|r_B\chi(B-A)+r_A\chi(A-B)|\le r_A+r_B\le4$.  By
Corollary~\ref{cor:hasse-weil},
$|K(A,B)| \le 2\sqrt\ell + 1$.  Hence
$|4|\Fp\setminus U_{a,b}| - \ell| \le 2\sqrt\ell + 7$, giving the result.
\end{proof}

For the primitive-root pair, the two individual set sizes add to exactly $\ell$,
so the union estimate immediately gives the corresponding intersection
estimate.

\begin{corollary}\label{cor:intersection}
Let $\gamma$ be a primitive root modulo $\ell$ and $N\in\Fps$.  Then
\[
  \left||H_\gamma(N;\ell) \cap H_{\gamma^2}(N;\ell)| - \frac{\ell}{4}\right|
  \le \frac{2\sqrt\ell+7}{4}.
\]
\end{corollary}

\begin{proof}
By Lemma~\ref{lem:one-target-size}, $\chi(\gamma N) = -\chi(N)$ (since
$\chi(\gamma)=-1$) and $\chi(\gamma^2 N) = \chi(N)$, so
$|H_\gamma|+|H_{\gamma^2}| = (\ell-\chi(N))/2 + (\ell+\chi(N))/2 = \ell$.
Inclusion-exclusion gives $|H_\gamma\cap H_{\gamma^2}| = \ell - |U_{\gamma,\gamma^2}|$,
and the bound follows from Corollary~\ref{thm:corrected-bound}.
\end{proof}

The exact primitive-root formula also identifies the condition for
the original pointwise inequality.

\begin{corollary}\label{cor:sufficient-condition}
Let $N\in\Fps$, let $\gamma$ be a primitive root modulo $\ell$, and put
$A=4\gamma N$, $B=4\gamma^2N$. Then the bound $|U_{\gamma,\gamma^2}(N;\ell)| \le 3\ell/4 + 1$
holds if and only if $K(A,B) + E(A,B) \ge -4$.  Uniformly in $N$, the same
bound is equivalent to the trace interval of
Corollary~\textup{\ref{cor:exact-trace-criterion}}.  Thus the bound for the primitive-root pair is a condition on the trace of the curve
$Y^2=(X^2-1)(X^2-\gamma)$, not a separate problem for each $N$.
\end{corollary}

\begin{proof}
From Corollary~\ref{cor:primitive-root-formula}, $|U_{\gamma,\gamma^2}| \le
3\ell/4+1$ iff $3\ell - K - E \le 3\ell+4$ iff $K+E \ge -4$.
\end{proof}

The proposed bound already fails modulo $5$.

\begin{proposition}\label{prop:counterexample-five}
The bound $|U_{\gamma,\gamma^2}| \le 3\ell/4+1$ fails for $\ell=5$, $N=1$,
$\gamma=2$.
\end{proposition}

\begin{proof}
The primitive root modulo $5$ is $\gamma=2$ (since $2^1=2,2^2=4,2^3=3,2^4=1$
generates all of $\mathbb{F}_5^*$).  Computing $H_2(1;5)$: for
$x=1,2,3,4$ we get $2\cdot1+1^{-1}=3$, $2\cdot2+2^{-1}=4+3=2$,
$2\cdot3+3^{-1}=6+2=3$, $2\cdot4+4^{-1}=8+4=2$ (all mod 5).
So $H_2(1;5) = \{2,3\}$.  For $H_4(1;5)$ ($b=\gamma^2=4$): $x=1$ gives
$4+1=0$, $x=2$ gives $8+3=1$, $x=3$ gives $12+2=4$, $x=4$ gives $16+4=0$
(mod 5), so $H_4(1;5) = \{0,1,4\}$.  The union is $\{0,1,2,3,4\} = \mathbb{F}_5$,
of size $5$.  But $3\ell/4+1 = 15/4+1 = 19/4 = 4.75 < 5$.
\end{proof}

\begin{remark}\label{rem:more-counterexamples}
The failure is not confined to $\ell=5$.  For $\ell=13$, $N=1$, $\gamma=7$
(one verifies $7$ has order $12$ modulo $13$, so it is a primitive root),
one computes $|H_7(1;13)\cup H_{49\bmod 13}(1;13)| = |H_7(1;13)\cup
H_{10}(1;13)| = 11 > 3\cdot13/4+1 = 10.75$.  These examples correspond to
elliptic curves $\mathcal{E}_{A,B}$ for which
$K(A,B) + E(A,B) < -4$.
\end{remark}

\section{Averages and fluctuations}
\label{sec:second-moment}

\begin{lemma}\label{lem:average-K}
Let $c \in \Fps$ be a quadratic nonresidue.  Then
\[
  \frac{1}{\ell-1}\sum_{A \in \Fps} K(A,cA) = -1.
\]
Equivalently, for every fixed $k\in\Fps$,
\[
  \frac{1}{\ell-1}\sum_{N \in \Fps}K(kN,ckN)=-1.
\]
\end{lemma}

\begin{proof}
We compute
\[
  \sum_{A\in\Fps}K(A,cA)
  =\sum_{z\in\Fp}\sum_{A\in\Fps}
    \chi\bigl((z^2-A)(z^2-cA)\bigr).
\]
For $z=0$, the inner sum is
\[
  \sum_{A\in\Fps}\chi(A\cdot cA)
  =\sum_{A\in\Fps}\chi(cA^2)=(\ell-1)\chi(c)=-(\ell-1).
\]
For $z\ne0$, substitute $A=z^2t$.  Since $z^2$ is a square,
\[
  \sum_{A\in\Fps}\chi((z^2-A)(z^2-cA))
  =\sum_{t\in\Fps}\chi((1-t)(1-ct)).
\]
Extending the last sum to all $t\in\Fp$ adds the $t=0$ contribution $1$.
The polynomial $(1-t)(1-ct)=ct^2-(1+c)t+1$ has nonzero discriminant
$(1-c)^2$, because $c\ne1$.  Lemma~\ref{lem:quadratic-shift-sum} applied to the quadratic polynomial gives $\sum_{t\in\Fp}\chi(ct^2-(1+c)t+1)=-\chi(c)=1$, since $c$ is a nonsquare.  Therefore the sum over $t\in\Fps$ is $1-1=0$.
Thus only $z=0$ contributes, and the total is $-(\ell-1)$.  The
$N$-form follows because multiplication by a fixed $k\in\Fps$ permutes
$\Fps$.
\end{proof}

\begin{lemma}\label{lem:average-E}
Let $c\in\Fps$ be a nonsquare and write
\[
  E(A,cA)=r_{cA}\chi(cA-A)+r_A\chi(A-cA),
  \qquad r_T=1+\chi(T).
\]
Then
\[
  \frac{1}{\ell-1}\sum_{A\in\Fps}E(A,cA)
  = (\chi(-1)-1)\chi(c-1).
\]
\end{lemma}

\begin{proof}
If $A$ is a square, then $r_A=2$ and $r_{cA}=0$, so
$E(A,cA)=2\chi(A(1-c))=2\chi(1-c)$.  If $A$ is a nonsquare, then $r_A=0$ and
$r_{cA}=2$, so $E(A,cA)=2\chi(A(c-1))=-2\chi(c-1)$.  There are
$(\ell-1)/2$ squares and $(\ell-1)/2$ nonsquares in $\Fps$.  Averaging gives
\[
  \chi(1-c)-\chi(c-1)=(\chi(-1)-1)\chi(c-1). \qedhere
\]
\end{proof}

\begin{proposition}\label{thm:average-union}
For any prime $\ell > 2$ and primitive root $\gamma$ modulo $\ell$,
\[
  \frac{1}{\ell-1}\sum_{N\in\Fps}|U_{\gamma,\gamma^2}(N;\ell)|
  = \frac{3\ell}{4} + \frac{1 - (\chi(-1)-1)\chi(\gamma-1)}{4}
  \le \frac{3\ell}{4} + \frac{3}{4}.
\]
\end{proposition}

\begin{proof}
Apply Corollary~\ref{cor:primitive-root-formula} with
$A=4\gamma N$ and $B=4\gamma^2N=cA$ with $c=\gamma$, a nonsquare.  Averaging
\[
  |U_{\gamma,\gamma^2}(N;\ell)|=\frac{3\ell-K(A,cA)-E(A,cA)}4
\]
over $N\in\Fps$, Lemma~\ref{lem:average-K} gives average $K=-1$, while
Lemma~\ref{lem:average-E} gives average $E=(\chi(-1)-1)\chi(\gamma-1)$.
The asserted identity follows.  Since $\chi(-1)-1\in\{0,-2\}$ and
$\chi(\gamma-1)\in\{-1,0,1\}$, the correction term is at most $2$ in
absolute value, giving the displayed upper bound.
\end{proof}

For a real-valued function $F$ on a finite nonempty set $S$, we use
$\mathbb E_S(F)=|S|^{-1}\sum_{x\in S}F(x)$ and
$\operatorname{Var}_S(F)=|S|^{-1}\sum_{x\in S}(F(x)-\mathbb E_S(F))^2$.
All averages here give equal weight to each nonzero field element.

The first moment gives the mean of $K(A,cA)$ as $A$ varies.  The
formula for the two square classes proved in Theorem~\ref{thm:two-valued-K} gives
the exact second moment immediately.  It depends on the integer $t_c=-1-K(1,c)$.

\begin{theorem}\label{thm:second-moment}
Let $c\in\Fps$ be a quadratic nonresidue, and let $t_c$ be the Frobenius trace
of
\[
  \mathcal E_c:\quad Y^2=(X^2-1)(X^2-c).
\]
Then
\[
  \frac{1}{\ell-1}\sum_{A\in\Fps}K(A,cA)^2=t_c^2+1.
\]
Equivalently, for every fixed $k\in\Fps$,
\[
  \frac{1}{\ell-1}\sum_{N\in\Fps}K(kN,ckN)^2=t_c^2+1.
\]
\end{theorem}

\begin{proof}
By Theorem~\ref{thm:two-valued-K}, the value of $K(A,cA)$ is $-1-t_c$ on the
$(\ell-1)/2$ squares in $\Fps$ and $-1+t_c$ on the $(\ell-1)/2$ nonsquares.
Therefore
\[
  \frac1{\ell-1}\sum_{A\in\Fps}K(A,cA)^2
  =\frac12(-1-t_c)^2+\frac12(-1+t_c)^2=t_c^2+1.
\]
For fixed $k\in\Fps$, multiplication by $k$ is a bijection on $\Fps$,
which gives the equivalent formula in $N$.
\end{proof}

\begin{corollary}\label{cor:concentration}
With $c$ and $t_c$ as in Theorem~\textup{\ref{thm:second-moment}},
\[
  \operatorname{Var}_{A\in\Fps}(K(A,cA))=t_c^2\le4\ell.
\]
Consequently, for every $\mu>0$,
\[
  \frac1{\ell-1}\#\{A\in\Fps: |K(A,cA)+1|\ge \mu\sqrt\ell\}
  \le \frac4{\mu^2}.
\]
In the primitive-root case, the same estimate gives
\[
  \frac1{\ell-1}\#\left\{N\in\Fps:
  \left||U_{\gamma,\gamma^2}(N;\ell)|-\frac{3\ell}{4}\right|
  \ge \frac{\mu\sqrt\ell+3}{4}\right\}
  \le \frac4{\mu^2}.
\]
\end{corollary}

\begin{proof}
Lemma~\ref{lem:average-K} gives $\mathbb E(K)=-1$, and
Theorem~\ref{thm:second-moment} gives $\mathbb E(K^2)=t_c^2+1$.  Hence
$\operatorname{Var}(K)=t_c^2$.  Theorem~\ref{thm:hasse-weil-genus-one} gives $|t_c|\le2\sqrt\ell$.
For the event $|K+1|\ge\mu\sqrt\ell$, its indicator is at most
$(K+1)^2/(\mu^2\ell)$. Averaging gives the displayed estimate for $K$.  Finally,
Corollary~\ref{cor:primitive-root-formula} gives
$|U_{\gamma,\gamma^2}|=(3\ell-K-E)/4$, and $|E|\le2$. If
$\bigl||U_{\gamma,\gamma^2}|-3\ell/4\bigr|
\ge(\mu\sqrt\ell+3)/4$, then $|K+E|\ge\mu\sqrt\ell+3$.
Since
\[
 |K+E|\le |K+1|+|E-1|\le |K+1|+3,
\]
this event implies $|K+1|\ge\mu\sqrt\ell$.
Multiplication by $4\gamma$ permutes $\Fps$, so the estimate for $A$
also bounds the stated proportion of $N$.
\end{proof}

\begin{remark}
The second moment depends on $t_c$.  For
example, if $t_c=0$, then $K(A,cA)=-1$ for every $A\in\Fps$; if
$|t_c|$ is close to $2\sqrt\ell$, the fluctuations are much larger.  The identity $\mathbb E(K^2)=1+t_c^2$ records the dependence on the
chosen ratio $c$ through $t_c$.
\end{remark}

\begin{remark}\label{rem:CM-case}
Suppose $\ell\equiv3\pmod4$ and $c=-1$. In the preceding proof the cubic
sum is $\sum_v\chi(v^3-v)$. The terms at $v$ and $-v$ cancel because
$\chi(-1)=-1$, and the term at $0$ vanishes. Thus $t_{-1}=0$ and
$K(A,-A)=-1$ for every $A\ne0$.

This gives an example for the pair of coefficients $a,-a$.
Put $A=4aN$ and $B=-A$. Exactly one of $A,B$ is a square. If $A$ is a
square, then $E(A,B)=2\chi(2A)=2\chi(2)$. If $A$ is a nonsquare, then
$B$ is a square and $E(A,B)=2\chi(-2A)=2\chi(2)$. In either case,
Theorem~\ref{thm:exact-union} gives
\[
 |H_a(N;\ell)\cup H_{-a}(N;\ell)|
 =\frac{3\ell+1-2\chi(2)}4.
\]
This size is independent of $N$ and is at most $3\ell/4+3/4$.
The pair $a,-a$ is a primitive-root pair $(\gamma,\gamma^2)$ only when
$\gamma=-1$ is a primitive root, which occurs only for $\ell=3$.
\end{remark}

\section{Unions of several sets and restricted images}
\label{sec:k-target}\label{sec:r-power}

We next bound a union of $k$ hyperbolic sieve sets.  The proof is a
straight inclusion--exclusion expansion of the complement indicator, with Weil
bounds applied only to squarefree character sums.  The boundary terms coming
from vanishing discriminants are treated explicitly.

\begin{theorem}\label{thm:k-target-explicit}
Fix $k\ge1$.  Let $\ell>2$ be prime, let $N\in\Fps$, and let
$a_1,\ldots,a_k\in\Fps$ be distinct.  Then
\[
  \left|\,\Bigl|\bigcup_{j=1}^k H_{a_j}(N;\ell)\Bigr|
        -\ell\Bigl(1-2^{-k}\Bigr)\right|
  \le (k-1+2^{-k})\sqrt\ell+k.
\]
\end{theorem}

\begin{proof}
Put $A_j=4a_jN$ and
\[
  \chi_j(z)=\chi(z^2-A_j),\qquad
  \delta_j(z)=\mathbf 1[z^2=A_j].
\]
By Lemma~\ref{lem:image-criterion}, the indicator of
$z\notin H_{a_j}(N;\ell)$ is
\[
  I_j(z)=\frac{1-\chi_j(z)-\delta_j(z)}2.
\]
Thus the complement $C$ of the union has size
\[
  |C|=2^{-k}\sum_{z\in\Fp}\prod_{j=1}^k(1-\chi_j(z)-\delta_j(z)).
\]
Expanding the product, the all-$1$ term contributes $\ell/2^k$.

Consider a nonempty pure-character term indexed by
$T\subseteq\{1,\ldots,k\}$:
\[
  \sum_{z\in\Fp}\prod_{j\in T}\chi(z^2-A_j)
  =\sum_{z\in\Fp}\chi\left(\prod_{j\in T}(z^2-A_j)\right).
\]
Since the $A_j$ are nonzero and distinct, the polynomial
$\prod_{j\in T}(z^2-A_j)$ is squarefree of degree $2|T|$.  Theorem~\ref{thm:weil-character-sum} therefore gives absolute value at most
$(2|T|-1)\sqrt\ell$.  There are $2^k-1$ nonempty subsets. Each index belongs to $2^{k-1}$
subsets, so $\sum_T|T|=k2^{k-1}$. Consequently
\[
 \sum_{T\ne\emptyset}(2|T|-1)
 =2k2^{k-1}-(2^k-1)=(k-1)2^k+1.
\]
The total absolute contribution of the pure-character terms is therefore
at most $((k-1)2^k+1)\sqrt\ell$.

It remains to bound terms containing at least one $\delta_j$.  If a product
contains two distinct factors $\delta_i\delta_j$, then it is identically zero,
because $A_i\ne A_j$ and a single $z$ cannot satisfy $z^2=A_i$ and
$z^2=A_j$ simultaneously.  Hence a nonzero boundary term contains exactly
one $\delta_j$, and arbitrary choices of character factors among the remaining
$k-1$ indices.  Such a term is supported on at most two roots of
$z^2=A_j$, and its absolute value is therefore at most $2$.  There are at
most $k2^{k-1}$ such terms.  Their total contribution before the factor
$2^{-k}$ is at most $k2^k$.

Combining the two preceding estimates,
\[
  \left||C|-\frac{\ell}{2^k}\right|
  \le 2^{-k}\bigl(((k-1)2^k+1)\sqrt\ell+k2^k\bigr)
  \le (k-1+2^{-k})\sqrt\ell+k.
\]
Since the union has size $\ell-|C|$, the same bound holds for the deviation of
$|\bigcup_j H_{a_j}|$ from $\ell(1-2^{-k})$.
\end{proof}

\begin{corollary}\label{cor:k-target-crt}
Let $k\ge1$ and $L>0$. Let $\ell_1,\ldots,\ell_n$ be distinct odd primes, each at least $L$, and none
dividing $N$.  Put $m=\ell_1\cdots\ell_n$.  For each $i$, choose distinct
coefficients $a_{i,1},\ldots,a_{i,k}\in\mathbb F_{\ell_i}^*$, and define the
CRT product set
\[
  V_m=\{s\in\mathbb Z/m\mathbb Z:
        s\bmod \ell_i\in\bigcup_{j=1}^k H_{a_{i,j}}(N;\ell_i)
        \text{ for every }i\}.
\]
Then
\[
  |V_m|
  \le m\prod_{i=1}^n\Bigl(1-2^{-k}+\left(k-1+2^{-k}\right)\ell_i^{-1/2}+k\ell_i^{-1}\Bigr)
  \le m\Bigl(1-2^{-k}+\left(k-1+2^{-k}\right)L^{-1/2}+kL^{-1}\Bigr)^n.
\]
Suppose that $N=uv$ for integers $u,v$. For every tuple $(j_1,\ldots,j_n)\in\{1,\ldots,k\}^n$, let
$a(j_1,\ldots,j_n)\in(\mathbb Z/m\mathbb Z)^*$ be the unique CRT class with
$a(j_1,\ldots,j_n)\equiv a_{i,j_i}\pmod{\ell_i}$ for every $i$.  Then the
residue of $a(j_1,\ldots,j_n)u+v$ modulo $m$ lies in $V_m$.
\end{corollary}

\begin{proof}
The CRT identifies $V_m$ with the Cartesian product of the local unions, so
\[
  |V_m|=\prod_{i=1}^n\Bigl|\bigcup_{j=1}^k H_{a_{i,j}}(N;\ell_i)\Bigr|.
\]
Applying Theorem~\ref{thm:k-target-explicit} locally gives the first bound,
and the second follows from $\ell_i\ge L$.

For the containment statement, fix $(j_1,\ldots,j_n)$.  Modulo $\ell_i$ we
have
\[
  a(j_1,\ldots,j_n)u+v\equiv a_{i,j_i}u+v\in H_{a_{i,j_i}}(N;\ell_i),
\]
because $(u,v)$ lies on the local hyperbola $xy=N$ modulo $\ell_i$.  Hence
the CRT residue lies in $V_m$.
\end{proof}

\begin{corollary}\label{thm:CRT}
Let $\ell_1,\ldots,\ell_n$ be distinct odd primes not dividing $N$, let
$m=\ell_1\cdots\ell_n$, and for each $i$ let $\gamma_i$ be a primitive root
modulo $\ell_i$.  The set $U_m\subseteq\Z/m\Z$ of classes whose projection
modulo $\ell_i$ lies in $U_{\gamma_i,\gamma_i^2}(N;\ell_i)$ for every $i$
satisfies
\[
  |U_m| = \prod_{i=1}^n |U_{\gamma_i,\gamma_i^2}(N;\ell_i)|
  \le m\prod_{i=1}^n\left(\frac{3}{4}+\frac{2\sqrt{\ell_i}+7}{4\ell_i}\right).
\]
If $N=uv$ with integers $u,v$, the set contains $au+v\bmod m$ for each of the $2^n$
coefficient classes satisfying $a\bmod\ell_i\in\{\gamma_i,\gamma_i^2\}$
for every $i$.  If the $\gamma_i$ are chosen by exact
local minimisation (Proposition~\textup{\ref{prop:verified-local}}), the product uses
the exact factors $|U_{\gamma_i,\gamma_i^2}|/\ell_i$.
\end{corollary}

\begin{proof}
Independence modulo distinct primes gives the product formula.  The displayed
upper bound is Corollary~\ref{thm:corrected-bound} applied at each prime.  The
covering statement is the $k=2$ case of Corollary~\ref{cor:k-target-crt} with
local coefficients $\gamma_i$ and $\gamma_i^2$.
\end{proof}

The hyperbolic sieve admits an exact strengthening when $N = u^rv$ and one
seeks the $r$-th-power divisor $u^r$.

\begin{definition}
Let $r \ge 2$ and $\ell \equiv 1 \pmod{r}$ be prime with $\gcd(N,\ell)=1$.
Define the \emph{$r$-power-constrained local hyperbola}
\[
  P_r(N;\ell) := \{(x,y) \in \Fps\times\Fps : xy \equiv N \pmod\ell,\;
                   x \in (\Fps)^r\},
\]
where $(\Fps)^r = \{t^r : t\in\Fps\}$ is the unique subgroup of $r$-th powers
in $\Fps$ (well-defined since $r\mid\ell-1$).
\end{definition}

\begin{proposition}\label{thm:r-power-local}
Assume $r\ge2$, $\ell$ is an odd prime, $r\mid\ell-1$, and $\ell\nmid N$. Then $|P_r(N;\ell)| = (\ell-1)/r$.
Moreover, if $N=u^rv$ for integers $u,v$, then
$(u^r\bmod\ell,\, v\bmod\ell) \in P_r(N;\ell)$.
\end{proposition}

\begin{proof}
Since $\Fps$ is cyclic of order $\ell-1$ and $r\mid\ell-1$, the subgroup
$(\Fps)^r$ has index $r$ and cardinality $(\ell-1)/r$.  For every
$x\in(\Fps)^r$, there is a unique $y = Nx^{-1}\in\Fps$ with $xy\equiv N$,
so $|P_r(N;\ell)| = |(\Fps)^r| = (\ell-1)/r$.  If $N=u^rv$, then $\ell\nmid u v$ because $\ell\nmid N$. Thus $u^r\in
(\Fps)^r$ and $(u^r)(v) = u^rv = N$, so $(u^r\bmod\ell,v\bmod\ell)\in
P_r(N;\ell)$.
\end{proof}

\begin{corollary}\label{cor:r-power-crt}
Let $r\ge2$ and $m=\ell_1\cdots\ell_n$ where the $\ell_i$ are distinct odd primes with
$\ell_i\equiv 1\pmod{r}$ and $\ell_i\nmid N$.  The CRT product set
\[
  P_r(N;m) := \{(x,y)\in(\Z/m\Z)^*\times(\Z/m\Z)^* : xy\equiv N,\;
               x\in((\Z/m\Z)^*)^r\}
\]
has cardinality $\prod_{i=1}^n\frac{\ell_i-1}{r} = O(m/r^n)$.
If $N=u^rv$ for integers $u,v$, then $(u^r\bmod m,\,v\bmod m)\in P_r(N;m)$.
\end{corollary}

\begin{proof}
The CRT isomorphism $(\Z/m\Z)^*\cong\prod_i\mathbb{F}_{\ell_i}^*$ preserves
both multiplication and the property of being an $r$-th power (since
$x=(x_1,\ldots,x_n)$ under CRT satisfies $x\in((\Z/m\Z)^*)^r$ iff each
$x_i\in(\mathbb{F}_{\ell_i}^*)^r$).  Hence $P_r(N;m)\cong\prod_i
P_r(N;\ell_i)$ under CRT, and the cardinality is the product of local
cardinalities from Proposition~\ref{thm:r-power-local}.
\end{proof}

\begin{corollary}\label{cor:r-power-linear}
Let $r\ge2$, let $\ell$ be an odd prime with $\ell\nmid N$ and $r\mid\ell-1$, and let $a\in\Fps$. The restricted image set
\[
  H_{a,r}(N;\ell) := \{ax + Nx^{-1} : x\in(\Fps)^r\}
\]
satisfies: $au^r+v\pmod\ell\in H_{a,r}(N;\ell)$ whenever $N=u^rv$ for integers $u,v$, and
$|H_{a,r}(N;\ell)|\le(\ell-1)/r$.  For a product $m=\ell_1\cdots\ell_n$ of distinct odd primes
$\ell_i\equiv1\pmod r$ not dividing $N$, and an integer coefficient $a$
coprime to $m$, the CRT candidate set for $au^r+v\pmod m$ has size at most $\prod_{i=1}^n\frac{\ell_i-1}{r} = O(m/r^n)$.
\end{corollary}

\begin{proof}
Substituting $x=u^r\in(\Fps)^r$ gives the containment.  The cardinality
bound holds because $H_{a,r}$ is the image of the domain $(\Fps)^r$ (of size
$(\ell-1)/r$) under $x\mapsto ax+Nx^{-1}$.  The CRT statement is the product
of local bounds.
\end{proof}

\begin{proposition}
\label{prop:exact-Har-size}
Let $r\ge2$, $r\mid\ell-1$, and $a,N\in\Fps$. Put $G=(\Fps)^r$ and define
\[
  H_{a,r}(N;\ell)=\{ax+Nx^{-1}:x\in G\}.
\]
Put $M=N/a\in\Fps$ and $h=|G|=(\ell-1)/r$, and write $G^2=\{x^2:x\in G\}$.  If $M\notin G$, then
\[
  |H_{a,r}(N;\ell)|=h.
\]
If $M\in G$, then
\[
  |H_{a,r}(N;\ell)|=\frac{h+f(M)}2,
  \qquad
  f(M)=\#\{x\in G:x^2=M\}.
\]
In particular,
\[
  \frac{\ell-1}{2r}\le |H_{a,r}(N;\ell)|\le \frac{\ell-1}{r}.
\]
More explicitly, if $h$ is odd then $f(M)=1$ for every $M\in G$, while if
$h$ is even then $f(M)=2$ when $M\in G^2$ and $f(M)=0$ otherwise.
\end{proposition}

\begin{proof}
Let $\phi:G\to\Fp$ be given by $\phi(x)=ax+Nx^{-1}$.  If
$\phi(x)=\phi(y)$, then
\[
  ax+Nx^{-1}=ay+Ny^{-1}.
\]
Multiplying by $xy$ and rearranging gives
\[
  (x-y)(axy-N)=0.
\]
Thus either $x=y$ or $xy=N/a=M$.  Therefore the only possible nontrivial
collision partner of $x$ is $\iota(x)=Mx^{-1}$.

If $M\notin G$, then $\iota(x)\notin G$ for every $x\in G$, because
$x^{-1}\in G$ and $Mx^{-1}\in G$ would imply $M\in G$.  Hence $\phi$ is
injective on $G$ and $|H_{a,r}|=h$.

If $M\in G$, then $\iota:G\to G$ is an involution, meaning that $\iota(\iota(x))=x$ for every $x\in G$.  Each value of $\phi$ has either one preimage, which is a fixed point
of $\iota$, or two preimages paired by $\iota$. There are $f(M)$ singleton
preimages and $(h-f(M))/2$ pairs, so the image has $(h+f(M))/2$ elements.  Its
fixed points are the solutions of $x^2=M$ in $G$, proving the formula.
Choose a generator $g$ of the cyclic group $G$. Write $M=g^j$ and
$x=g^i$. Then $x^2=M$ is equivalent to $2i\equiv j\pmod h$.
If $h$ is odd, $2$ is invertible modulo $h$ and there is exactly one
solution. If $h$ is even, there are two solutions when $j$ is even,
and none when $j$ is odd. The condition that $j$ is even is precisely
$M\in G^2$. Finally, every value has at most two preimages, so
$h/2\le |H_{a,r}(N;\ell)|\le h$, as asserted.
\end{proof}

\begin{remark}
For one coefficient, the restricted image gives at most $(\ell-1)/r$
candidate residues. For $k$ coefficients the elementary union bound gives
at most $k(\ell-1)/r$ residues, and hence at most $m(k/r)^n$ in the CRT
product. In particular, the bound $O(m/r^n)$ for one coefficient does not
extend to a union of several images without a factor for their number or
an additional overlap estimate.
\end{remark}

\section{Algorithmic verification of local sieve sets}
\label{sec:algorithmic}\label{sec:paths-forward}

The preceding sections give asymptotic and worst-case bounds.  For the small
primes used in an actual CRT sieve, one can compute the exact local sets and
optimise the choice of coefficients.

\begin{proposition}\label{prop:verified-local}
Fix an odd prime $\ell\nmid N$ and a nonempty finite coefficient set $\mathcal{A}\subset\Fps$.
The union
\[
  U_{\mathcal{A}}(N;\ell):=\bigcup_{a\in\mathcal{A}}H_a(N;\ell)
\]
can be computed exactly in $O(|\mathcal{A}|\,\ell)$ field operations.
Hence, for any explicitly listed nonempty family $\mathfrak{A}$ of nonempty coefficient sets,
the minimisation
\[
  \min_{\mathcal{A}\in\mathfrak{A}}|U_{\mathcal{A}}(N;\ell)|
\]
can be solved exactly in $O\bigl(\ell\sum_{\mathcal{A}\in\mathfrak{A}}
|\mathcal{A}|\bigr)$ field operations.
\end{proposition}

\begin{proof}
First compute all $x^{-1}$ for $x\in\Fps$ in $O(\ell)$ field operations,
where a field operation is an addition, subtraction, multiplication, or inversion
in $\Fp$. For each $a\in\mathcal{A}$ and $x\in\Fps$, compute $ax+Nx^{-1}\bmod\ell$ and
mark the result in a boolean array of length $\ell$, meaning an array
with entries $0$ or $1$.  After processing all $a$,
count the marked entries.  This gives $U_{\mathcal{A}}(N;\ell)$ exactly in
$O(|\mathcal{A}|\ell)$ field operations.  The minimisation is resolved by repeating
for each $\mathcal{A}\in\mathfrak{A}$ and comparing.
\end{proof}

\begin{proposition}
\label{prop:trace-verification}
Fix an odd prime $\ell$ and a primitive root $\gamma$ modulo $\ell$.  Let
$t_\gamma$ be the Frobenius trace of
\[
  \mathcal E_\gamma:\quad Y^2=(X^2-1)(X^2-\gamma).
\]
Then the maximum and minimum of $|U_{\gamma,\gamma^2}(N;\ell)|$ as $N$ ranges
over $\Fps$ are the two explicit values in
Theorem~\textup{\ref{thm:primitive-root-two-values}}.  In particular, the proposed
local bound
\[
  |U_{\gamma,\gamma^2}(N;\ell)|\le \frac{3\ell}{4}+1
\]
for all $N\in\Fps$ can be verified by checking the single trace inequality
\[
  -3+2\chi(\gamma-1)
  \le
  t_\gamma
  \le
  3+2\chi(1-\gamma).
\]
\end{proposition}

\begin{proof}
This is exactly Theorem~\ref{thm:primitive-root-two-values} together with
Corollary~\ref{cor:exact-trace-criterion}.
\end{proof}

\begin{remark}
Let $\phi(d)$ be Euler's totient function, the number of integers from $1$
to $d$ coprime to $d$. For an explicitly listed collection of primitive
roots, Proposition~\ref{prop:trace-verification} reduces verification for
all $N\in\Fps$ to one character sum for each root. Precomputing the
Legendre symbols of all residues takes $O(\ell)$ field operations by
marking the nonzero squares. Each trace can then be computed in $O(\ell)$
field operations, giving $O(\phi(\ell-1)\ell)$ operations for the list.

Alternatively, one can compute the traces by counting points on elliptic
curves. The proof of Theorem~\ref{thm:two-valued-K} shows that $t_\gamma$
is also the trace of the cubic equation
\[
 Y^2=X(X-1)(X-\gamma),
\]
with its one point at infinity. Indeed, its point count is
\[
 \ell+1+\sum_{x\in\Fp}\chi(x(x-1)(x-\gamma))
 =\ell+1-t_\gamma.
\]
Since $\gamma\ne0,1$, the cubic has distinct roots and defines a
nonsingular elliptic curve. Schoof's algorithm~\cite{Schoof1985} computes its number of points,
and hence $t_\gamma$, deterministically in a number of bit operations
polynomial in $\log\ell$. Thus, for an explicitly supplied list of all
primitive roots, the trace tests take
$
 O\bigl(\phi(\ell-1)(\log\ell)^C\bigr)
$
bit operations for some absolute constant $C$. This bound concerns computing
and testing the traces; it does not include constructing the list of
primitive roots. The case $\ell=3$ can be checked directly.

For the restricted images, membership in $(\Fps)^r$ is equivalent to
$x^{(\ell-1)/r}=1$ when $r\mid\ell-1$; repeated squaring checks this in
$O(\log\ell)$ field multiplications.
\end{remark}

The formulas above give several ways to choose and bound the candidate sets.
One may choose $\gamma$ adaptively: among the $\phi(\ell-1)$ primitive roots
modulo $\ell$, search for a value whose trace $t_\gamma$ lies in the interval
of Corollary~\textup{\ref{cor:exact-trace-criterion}}.  This is stronger than testing
individual $N$'s, because one trace condition verifies the proposed pointwise
bound for all $N\in\Fps$.  At the CRT level, even if
$|U_{\gamma,\gamma^2}|$ exceeds $3\ell/4+1$ at an individual prime
$\ell_i$, the relevant quantity is the product of the selected local factors.
Corollary~\textup{\ref{thm:CRT}} gives the CRT product bound obtained from the
uniform estimate, and Proposition~\textup{\ref{prop:verified-local}} explains
how to replace those uniform factors by the exactly verified quantities
$|U_{\gamma_i,\gamma_i^2}(N;\ell_i)|/\ell_i$ at the chosen primes.
Finally, one can use more linear forms: the union of $k$ sets
has size $\ell(1-2^{-k})+O_k(\sqrt\ell)$ by
Theorem~\ref{thm:k-target-explicit}.  This covers more local coefficient
choices at the cost of a larger local candidate set. For fixed $k$, the
upper bound is strictly smaller than $\ell$ whenever
\[
 (k-1+2^{-k})\ell^{-1/2}+k\ell^{-1}<2^{-k}.
\]
At primes satisfying this inequality, the product bound is strictly
smaller than the full modulus. For two sets, the sharper bound in
Corollary~\ref{thm:corrected-bound} is smaller than $\ell$ for every
prime $\ell\ge17$, since then $2\sqrt\ell+7<\ell$.

\section{Conclusion}
\label{sec:conclusion}

For distinct nonzero coefficients $a,b$, the union
$H_a(N;\ell)\cup H_b(N;\ell)$ has size $3\ell/4+O(\sqrt\ell)$,
with an exact formula in terms of $K(4aN,4bN)$. Thus a common candidate
set for two linear forms excludes about one quarter of the residues at a
large prime. The CRT construction combines these restrictions at distinct
primes and gives the product bounds stated in the paper.

For a fixed ratio $c$, the identity
$K(A,cA)=-1-\chi(A)t_c$ reduces the dependence on $N$ to a square or
nonsquare choice. In the primitive-root case this gives explicit union sizes
and an exact trace interval for the bound $3\ell/4+1$ to hold for every
nonzero $N$ modulo $\ell$. The examples modulo $5$ and $13$ show why that
stronger bound cannot be assumed without checking the trace condition.

For $k$ distinct coefficients, the union has size
$\ell(1-2^{-k})+O_k(\sqrt\ell)$ with the explicit error in
Theorem~\ref{thm:k-target-explicit}. Here the subscript $k$ means that the
constant may depend on $k$. For inputs $N=u^rv$, restricting the first
coordinate to $r$-th powers gives exact local counts, including the image
formula in Proposition~\ref{prop:exact-Har-size}.

Further questions include the distribution of $t_\gamma$ over primitive
roots, the existence of primitive roots satisfying the trace interval, and
how the local candidate sets can be used in deterministic factorization
algorithms. The bounds on set sizes do not by themselves establish a faster
factorization algorithm: the construction costs and the subsequent search
must also be analysed.

\vspace{0.6cm}
\noindent{\bf Acknowledgments.}
The authors would like to thank the editor for the prompt handling of our paper and the reviewers for the comments which improved the quality of our work.
Part of this paper was written while the first-named author (PS) visited NORCE and University of Bergen in Spring of 2026. He thanks the institutions for the excellent working conditions.

\appendix
\section{Python verification of the union-size counterexamples}
\label{app:verification}

The following short script verifies the two numerical counterexamples from Section~\ref{sec:results}.  It is included only as a reproducibility aid; none of the proofs above depends on computational verification.
\begin{verbatim}
def H(a, N, l):
    return {(a*x + N*pow(x, -1, l)) % l for x in range(1, l)}

# First counterexample: l=5, N=1, gamma=2
print(H(2,1,5) | H(4,1,5))        # {0,1,2,3,4} -- full F_5
print(3*5/4 + 1)                   # 4.75 < 5: bound fails

# Second counterexample: l=13, N=1, gamma=7
print(len(H(7,1,13) | H(10,1,13))) # 11 > 3*13/4+1 = 10.75
\end{verbatim}

\section{SageMath script for Frobenius trace verification}
\label{app:sage}

The following SageMath script computes the Frobenius trace $t_\gamma$ of the
curve $\mathcal{E}_\gamma\colon Y^2=(X^2-1)(X^2-\gamma)$ over $\mathbb{F}_\ell$
for a given primitive root $\gamma$, evaluates the admissible trace interval
of Corollary~\ref{cor:exact-trace-criterion}, and verifies the two-case
formula of Theorem~\ref{thm:primitive-root-two-values} for every
$N\in\mathbb{F}_\ell^*$.  Running it on the two counterexamples recovers the
values recorded in Remark~\ref{rem:trace-verify}.
\begin{verbatim}
def frobenius_trace(ell, gamma):
    """Frobenius trace of Y^2 = (X^2-1)*(X^2-gamma) over GF(ell)."""
    F = GF(ell)
    count = sum(1 + kronecker_symbol((x^2 - 1)*(x^2 - gamma), ell)
                for x in range(ell))
    return ell + 1 - (count + 2)   # +2 for the two points at infinity

def trace_interval(ell, gamma):
    """Admissible trace interval from the trace criterion."""
    lo = -3 + 2*kronecker_symbol(gamma - 1, ell)
    hi =  3 + 2*kronecker_symbol(1 - gamma, ell)
    return lo, hi

def verify_two_valued(ell, gamma, t):
    """Check two-case formula (the union-size formula) for all N in F_ell^*."""
    g2 = (gamma^2) % ell
    chi1g = kronecker_symbol(1 - gamma, ell)
    chig1 = kronecker_symbol(gamma - 1, ell)
    for N in range(1, ell):
        A = (4 * gamma * N) % ell
        chi_A = kronecker_symbol(A, ell)
        if chi_A == 1:
            expected = (3*ell + 1 + t - 2*chi1g) // 4
        else:
            expected = (3*ell + 1 - t + 2*chig1) // 4
        Ha = {(gamma*x + N*inverse_mod(x, ell)) % ell for x in range(1, ell)}
        Hb = {(g2*x   + N*inverse_mod(x, ell)) % ell for x in range(1, ell)}
        actual = len(Ha | Hb)
        if expected != actual:
            return False, N, expected, actual
    return True, None, None, None

# ---- Run on both counterexamples ----
for (ell, gamma) in [(5, 2), (13, 7)]:
    t   = frobenius_trace(ell, gamma)
    lo, hi = trace_interval(ell, gamma)
    ok, bad_N, exp, act = verify_two_valued(ell, gamma, t)
    print(f"ell={ell}, gamma={gamma}: t_gamma={t}, "
          f"interval=[{lo},{hi}], in interval={lo<=t<=hi}")
    print(f"  Two-valued formula verified for all N: {ok}")
\end{verbatim}

\noindent
Expected output:
\begin{verbatim}
ell=5,  gamma=2: t_gamma=-2, interval=[-1,5], in interval=False
  Two-valued formula verified for all N: True
ell=13, gamma=7: t_gamma=-6, interval=[-5,1], in interval=False
  Two-valued formula verified for all N: True
\end{verbatim}

\end{document}